\documentclass[a4paper, 11 pt, twoside]{amsart}
\usepackage{srollens-en}[2011/03/03]
\usepackage[left = 3.5cm, right = 3.5cm, headsep = 6mm,
footskip = 10mm, top = 35mm, bottom = 35mm, footnotesep=5mm, headheight =
2cm]{geometry}

\usepackage[usenames, dvipsnames]{xcolor}

\usepackage{tikz, tikz-cd}
\usepackage{cancel}
\usepackage[unicode,bookmarks, pdftex, backref = page]{hyperref}
\hypersetup{colorlinks=true,citecolor=NavyBlue,linkcolor=NavyBlue,urlcolor=Orange, pdfpagemode=UseNone, breaklinks=true}

\usepackage{enumitem}
\setlist[description]{leftmargin=0cm,  labelindent=\parindent}
\DeclareMathOperator{\mult}{mult}

\DeclareFontFamily{OT1}{rsfs}{}
\DeclareFontShape{OT1}{rsfs}{n}{it}{<-> rsfs10}{}
\DeclareMathAlphabet{\curly}{OT1}{rsfs}{n}{it}

\usepackage{changes}

\newcommand{\bC}{\ensuremath{\mathbb{C}}}
\newcommand{\deq}{\stackrel{\text{def}}{=}}

\title[Seshadri constants on multiple planes]{Simple cyclic covers of the plane and  Seshadri constants of some general hypersurfaces in weighted projective space}

\author{S\"onke Rollenske}
\address{S\"onke Rollenske\\FB 12/Mathematik und Informatik\\
Philipps-Universit\"at Marburg\\
Hans-Meerwein-Str. 6\\
35032 Marburg\\
Germany}
\email{rollenske@mathematik.uni-marburg.de}
\author{Alex K\"uronya}
\address{Alex K\"uronya \\ Institut f\"ur Mathematik\\ Goethe Universit\"at Frankfurt \\ Robert-Mayer-Str. 6--10.\\ D-60325 Frankfurt am Main \\ Germany}
\email{kuronya@math.uni-frankfurt.de}

\dedicatory{Dedicated to Fabrizio Catanese on the occasion of his 70th birthday.}

\begin{document}

\begin{abstract}
Let $X \subset \IP(1,1,1,m)$ be a general hypersurface of degree $md$ for some  for $d\geq 2$ and $m\geq 3$.
We prove that  the  Seshadri constant $\epsilon ( \ko_X(1), x)$ at a general point $x\in X$ lies in the interval $\left[\sqrt{d}- \frac d m, \sqrt{d}\right]$ and thus approaches the possibly irrational number $\sqrt d$ as $m$ grows. 

The main step is a detailed study of the case where $X$ is a simple cyclic cover of the plane. 
\end{abstract}

\maketitle
\setcounter{tocdepth}{1}
\tableofcontents

\section*{Introduction}
 
 Seshadri constants measure local positivity of line bundles. In this capacity they have close ties with two major conjectures  in the theory  of smooth complex surfaces: the Segre--Harbourne--Gimigliano--Hirschowitz conjecture \cite{G,DKMS,Ha,Hi,S} and the Bounded Negativity Conjecture \cite{11authors,NCAS}. 
 
 A big line bundle is locally positive at a point, if the global sections of a suitable multiple of it embed an open neighbourhood of the point into projective space \cite{BCL,KL_GeomAsp}. If this is the case then one can try to measure how many 'reasonably different' global sections there are around the reference point. 
 
 One way to quantify local positivity is via Seshadri constants defined for ample line bundles by Demailly \cite{Dem} during his first attempt at Fujita's global generation conjecture.  For a thorough introduction and references we refer the reader to \cite{Bauer99} and  \cite[Chapter 5]{PAGI}. 
 
 Rationality of Seshadri constants is a long-standing open question, and although it is widely assumed to be false, there is no counterexample even in higher dimensions. As examplified by the volume or asymptotic cohomology of a line bundle \cite{BKS} or Newton--Okounkov bodies \cite{KLM},  asymptotic invariants of linear series have a tendency to be rational on surfaces and wildly irrational in higher dimensions. The existing evidence in the case of Seshadri constants is very mixed. 
 
 On the one hand all computed cases of Seshadri constants are rational, and this includes all abelian surfaces, some of which have round nef cones. 
 At the same time  there is no structural result pointing towards rationality, although the  rationality of Seshadri constants would disprove Nagata's conjecture \cite{DKMS}, which is more or less unanimously expected to be true. 
 
 Seshadri constants on general surfaces of degree $d\geq 5$ in projective space are expected  to be integral multiples of $\sqrt{d}$, hence often irrational. There is a closely related  work  centering around Szemberg's conjecture \cite{CrGp,Sz} on primitive solutions to Pell's equation and optimal lower bounds on Seshadri constants. 	
 
Here we initiate a study of a  class of surfaces that avoided being studied from this point of view. 

\begin{custom}[Theorem A] Fix integers $d\geq 2$ and $m\geq 3$. 
 Let $X$ be a smooth hypersurface of degree $md$ in $\IP(1,1,1,m)$ not passing through the singular point. Let $L = \ko_X(1)$, and $x$ a point on $X$. If the pair $(X, x)$ is general then 
 \begin{equation} \label{eq: estimate} 
\sqrt{d}- \frac d m \leq \epsilon(L;x) \leq \sqrt{d} = \sqrt{L^2}.\tag{$\ast$}  
 \end{equation}
In particular, for every integer $d$ and $\delta>0$ there exists  a point $x$ on a polarised surface $(X,L)$ such that $L^2 = d$ and  $\epsilon(L;x) \in [\sqrt{L^2}-\delta, \sqrt{L^2}]$.
\end{custom}

The main step is to prove the estimate in the special case when   $B$ is a very general plane curve of degree $md$,  $X = X_{d,m}\to \IP^2$ the simple cyclic $d$-uple cover branched over $B$ and $L  = \pi^*\mathcal O_{{\mathbb P}^2}(1)$. One can then deduce  \eqref{eq: estimate} directly for a very general point of $X$ using a result of Bauer \cite{Bauer99} and the following generalisation of a result of Cox \cite{cox}, which might be of independent interest.

\begin{custom}[Theorem B]
  	A very general  simple cyclic cover $X$ of the plane  is a smooth surface with Picard number $\rho(X) = 1$.
\end{custom}

% 
% \sr{old version}
% {
% \scshape
% We will consider very general, simple cyclic multiple planes and make use of the control that comes from the covering situation. 	
% 
% Concretely, let $d\geq 2$ and $m\geq 3$ be integers and $B$ a very general plane curve of degree $md$. Let  $X = X_{d,m}\to \IP^2$ be the simple cyclic $d$-uple cover branched over $B$ and $L  = \pi^*\mathcal O_{{\mathbb P}^2}(1)$. 
%   \sr{Changed order of theorems}
% \begin{custom}[Theorem A]
%   	The Seshadri constant of $X=X_{d,m}$ at a very general point $x\in X$  satisfies
%   	\[ \sqrt{d}- \frac d m \leq \epsilon(L;x) \leq \sqrt{d} = \sqrt{L^2}.\]
%   	In particular, for every integer $d$ and $\epsilon>0$ there exists  a point $x$ on a polarised surface $(X,L)$ such that $L^2 = d$ and  $\epsilon(L;x) \in [\sqrt{L^2}-\epsilon, \sqrt{L^2}]$.
% \end{custom}
% While our result shows that on simple cyclic multiple planes Seshadri constants can at least be arbitrarily close to irrational we were so fare unable to decide if they can be irrational.
% 
% One of the main ingredients of the proof is the following slight generalisation of a result of Cox \cite{cox}.
% 
% \begin{custom}[Theorem B]
%   	A very general  simple cyclic cover $X$ of the plane  is a smooth surface with Picard number $\rho(X) = 1$.
% \end{custom}
% }
%   
 
\subsection*{Acknowledgements.} We are grateful to Thomas Bauer and Jakob Stix for helpful discussions.  This cooperation started during the \emph{Workshop on local negativity and positivity on algebraic surfaces}  in 2017 in Hannover, we thank the organizers Roberto Laface and Piotr Pokora  for the invitation. The first-named author was partially supported by the LOEWE grant 'Uniformized structures in Arithmetic and Geometry'.  

 \subsection*{Conventions.}
We work over the complex number field. A variety is an integral separated scheme of finite type over $\bC$. A polarized variety $(X,L)$ is a projective variety $X$ equipped with an ample line bundle $L$. 
 
\section{Seshadri constants and submaximal curves}
 
 Here we collect the necessary material about Seshadri constants on surfaces. Good references for this topic are \cite{Bauer99} and  \cite[Chapter 5]{PAGI}. 
 Let $(X,L)$ be a polarised  algebraic surface and $C\subset X$ be a curve containing a point $x$. We then define 
 \[
 \epsilon_{C,x}:= \frac{C.L}{\mult_x(C)}
 \]
 and call $C$ a submaximal curve for $L$ at $x$ if $\epsilon_{C, x}<\sqrt{L^2}$. The \emph{Seshadri constant} $\epsilon(X,L;x)$ (often just $\epsilon(L;x))$ of $L$ at the point $x\in X$ is defined as 
 \begin{equation}\label{eq: Seshadri via nefness}
 \epsilon(X,L;x) \deq \inf_{x\in C} \epsilon_{C,x}\ . 
 \end{equation}
 Equivalently, upon writing $\pi\colon \widetilde{X}\to X$ for the blowing-up of $X$ at $x$ with exceptional divisor $E$, one has
 \[
 \epsilon(L;x) = \sup\{ t>0\,\mid\, \pi^*L-tE \text{ is nef} \}\ . 
 \]
 The definition gives rise to the following quick upper bound:
 \[
 \epsilon(L;x) \leq \sqrt{(L^2)}\ .
 \]
 It is known that if the inequality is sharp then there must exist a submaximal curve $C\subseteq X$. Equivalently, submaximal means that the proper transform $\widetilde{C}$ of $C$ on $\widetilde{X}$ satisfies $((\pi^*L-\epsilon(L;x)E)\cdot \widetilde{C}) = 0$. 
  
 A couple of words about the behaviour of $\epsilon(L;x)$ in its two arguments. As can be seen from its definition as nef threshold on the blow-up, $\epsilon(L;x)$ only depends on the numerical equivalence class of $L$, and 
 \[
 \epsilon(mL;x) = m\cdot \epsilon(L;x)
 \]
 for all positive integers $m$. In particular, if $\rho(X)=1$, then $\epsilon(L;x)$ for the ample generator $L$ of $X$ determines all Seshadri constants at $x\in X$. 
 
 By \cite{Oguiso} Seshadri constant are upper-semicontinuous in $x$. Consequently, for given $L$, the numbers $\epsilon(L;x)$ are constant for very general $x\in X$.
 It is customary to write $\epsilon(L;1)$ for this common value. If $\rho(X)=1$ then we will denote by $\epsilon(X, 1)$ the Seshadri constant of the ample generator of $X$  at a very general point of $X$.	In this case we will call $\epsilon(X;1)$ the \emph{Seshadri constant of the surface $X$}. 
 
Usually one is especially interested in irreducible submaximal curves but it is sometimes convenient to allow $C$ to be any curve containing $x$. 

% We denote by $\epsilon(X, 1)$ the Seshadri constant at a very general point of $X$.
% 
% \begin{prop}[Uniqueness of submaximal curves]
% 	\end{prop}
% 

Let $H$ be the class of a line in the plane. The proof of the following is an elementary computation using the projection formula.
\begin{lem}\label{lem: no pullback}
	Let $\pi\colon X\to \IP^2$ be a $d$-uple plane (that is, a finite surjective morphism of degree $d$),  with polarisation $L=\pi^*H$. If $C$ is a plane curve then $\pi^*C$ is not submaximal for any point on $X$. 
\end{lem}

A key ingredient in our argument is the following result of Thomas Bauer, that controls the degree of a submaximal curve. 
\begin{thm}[\protect{Bauer,  \cite[Thm.~4.1]{Bauer99}}]\label{thm: Bauerbound}
	Let $(X,L)$  be a polarised surface and $x\in X$ a very general point. 
	If $C$ is an irreducible submaximal curve for $x$, then 
	\[ L.C < \frac{L^2}{\sqrt{L^2}- \epsilon_{C,x}}.\]
\end{thm}

\section{Simple cyclic multiple planes}
A $d$-uple plane is a finite map $\pi\colon X \to \IP^2$ of degree $d\geq 2$. In general, the structure of such maps is quite complicated if $d>2$ but we only discuss the following simple construction, where $X$ arises by taking the $d$-th root of a section of a  line bundle.  More precisely, let $m\geq 1$ and $B$ be a plane curve of degree  $md$ defined by an equation $f$. 
Then a simple cyclic cover  $X$ is a hypersurface of degree $md$ in $\IP(1,1,1,m)$, with equation 
\begin{equation}\label{eqn:hypsurf in wps}
w^d+f(x,y,z)=0\ .
\end{equation}
Note that $\pi\colon X \to \IP^2$ is a Galois ramified cover with group $\IZ/d\IZ$. Decomposing $\pi_*\ko_X$  into eigensheaves with respect to this action we get a decomposition 
\[\pi_*\ko_X =\bigoplus_{k = 0}^{d-1} \ko_{\IP^2}(-km) = \left( \bigoplus _{k\geq 0 } \ko_{\IP^2}(-km) w^k \right) / (w^d - f),\]
where the first decomposition is as $\ko_{\IP^2}$-module while the second specifies in addition the algebra structure. A local computation shows that $X$ is smooth if and only if $B$ is smooth. 

For a more complex analytic point view on simple cyclic covers see \cite[Ch.~1]{BHPV}, for a more general theory of abelian covers compare \cite{pardini91}.

\begin{defin}
 A simple cyclic $d$-uple plane  is a  ramified cover  
 \[\pi\colon X=X_{d,m}\to \IP^2\] of degree $d$, branched over over a curve $B$ of degree  $dm$ arising as above. 
 We say $X$ is very general if $B$ is very general.
\end{defin}
% 
% 
% 
% \begin{rem}
% Note that $\pi_*\ko_X = \bigoplus_{k = 0}^{d-1} \ko_{\IP^2}(-km)$ as a vector bundle on $\IP^2$.
%  As an $\ko_{\IP^2}$-algebra we can describe it via the surjection
%  \[\ka =  \bigoplus _{k\geq 0 } \ko_{\IP^2}(-km) w^k \onto \ka/(w^d - f) = \pi^*\ko_X.\]
%  Taking the relative spectrum we see that $X$ is embedded into the total space of the line bundle $\ko_{\IP^2}(m)$ as a hypersurface, which in turn is the complement of the singular point in $\IP(1,1,1,m)$. 
% \end{rem}

For lack of convenient reference we sketch a proof of the following.
\begin{prop}[Theorem B] Let $d\geq2$ and $m\geq 3$ be integers. Then a very general simple cyclic $d$-uple plane $\pi\colon X \to \IP^2$ with branch curve $B$ of degree $dm$ has Picard rank $1$. 
 \end{prop}
 This is proved for the full family of weighted hypersurfaces in \cite{cox} and we follow the proof very closely. The case of degree $2$ is covered in \cite{Buium}.
 We will consider $d$-uple planes $X_{d,m}$ as hyperplanes in $\IP(1,1,1,m)$ with equation (\ref{eqn:hypsurf in wps}). As explained in the first paragraph of \cite[Proof of Theorem]{cox} (see also  \cite[III (a)]{CGGH}), in order to verify that  $\rho(X_{d,m})=1$ for the general hypersurface, it suffices to check that the map
 \[
 H^1(X,\Theta_X)_0\otimes H^{2,0}(X) \longrightarrow H^{1,1}(X)
 \]
 is  surjective, where $H^1(X,\Theta_X)_0$ stands for the image of $H^1(X,\Theta_X)$ under the Kodaira--Spencer map of the family of degree $d$ hypersurfaces in $\IP(1,1,1,m)$ with equation (\ref{eqn:hypsurf in wps}). The case of double covers is treated in \cite{Buium}.
 
 As far as the link between various cohomology spaces occurring above  and  the graded polynomial ring $S\deq \bC[x,y,z,w]$ goes, write
 \[
 R \deq  S/(\nabla F) = S/(w^{d-1}, \nabla f)\ ,
 \]
 then 
 \begin{eqnarray*}
 	H^{2,0}(X) & \simeq  & R^{dm-(1+1+1+m)} = R^{dm-m-3}\ ,\\
 	H^{1,1}(X) & \simeq  & R^{2dm-m-3}\ , \text{ and }\\
 	H^{1}(X,\Theta_X)_0 & \simeq  & R^{dm}\ .
 \end{eqnarray*}
 For proofs see the reference in \cite{cox}. 

\begin{proof}
Consider the weighted polynomial rings $S'= \IC[x,y, z]\subset S = \IC[x,y,z,w]$, where $x,y,z$ have degree $1$ and $w$ has degree $m$. 
 Let $f$ be the equation defining $B$ so that $X$ is cut out by $F=w^d+f\in S = \IC[x,y,z,w]$. Consider the Milnor algbra 
 \[R = S/(\nabla F) = S/(w^{d-1}, \nabla f),\]
 which is the homomorphic image of $T = S/(w^{d-1})$.
 Let $V$ be the tangent space to the family of $d$-uple planes at $X$. Then as in \cite{cox} the result follows if the composition
\[\begin{tikzcd} V\tensor H^{2,0}(X) \rar{\kappa\tensor\id} &  H^1(\Theta_X)\tensor H^{2,0}(X) \rar& H^{1,1}(X)
   \end{tikzcd}\]
is surjective, where $\kappa$ is the Kodaira-Spencer map. Considering $d$-uple planes as particular hypersurfaces and argueing as in \cite{cox}, this multiplication map fits into a commutative  diagram
\[
 \begin{tikzcd}
  S'_{md}\tensor T_{md-m-3} \rar\dar[twoheadrightarrow] & T_{2md-m-3}\dar[twoheadrightarrow]\\
  V \tensor R_{md-m-3}\dar \rar & R_{2md-m-3}\dar[equal]\\
  R_{md} \tensor R_{md-m-3}\dar{\isom} \rar & R_{2md-m-3}\dar{\isom}\\
   H^1(\Theta_X)\tensor H^{2,0}(X) \rar& H^{1,1}(X)
 \end{tikzcd}
\]
Now note that as a graded  $S'$-module $T_k = \bigoplus_{i=0}^{d-2} w^i\cdot S'_{k-im}$
 and all summands on the right hand side are non-zero as soon as $k\geq m(d-2)= m(d-1)-m$. Since the multiplication maps in the standard polynomial ring $S'$ are surjective, the surjectivity of the first row of the diagram, and hence our claim, follows if $m\geq 3$. 
\end{proof}

We would like to have that the Picard group is actually generated by $L$, at least for large $m$,  but we do not know a convenient reference for this. For our purpose the following simple observation is enough. 
\begin{cor}\label{cor: dC}
Let $\pi \colon X \to \IP^2$ be a very general simple cyclic $d$-uple plane and $L  = \pi^*\mathcal O_{{\mathbb P}^2}(1)$.
If $C$ is any curve on $X$ then there exists a $k$ such that $dC \in |kL|$. 
\end{cor}
\begin{proof}
 Let $\sigma$ be a generator of the Galois group of $\pi$. Since the Picard rank is $1$, the automorphisms act trivially on the Picard group  and we have 
 \[d C \sim \sum _{k = 1}^d (\sigma^k)^*C = \pi^*\pi_*C \in |\pi^*\ko_{\IP^2}(k)| = |kL|\]
 for some $k$. 
\end{proof}

\section{Proof of Theorem A}

We start by proving the estimate \eqref{eq: estimate} for a very general  simple cyclic $d$-uple plane $\pi\colon X \to \IP^2$ branched over a curve of degree $md$. 

 Assume  $C$ is an irreducible curve which is submaximal for a very general point $x\in X$ and such that $\sqrt{L^2} - \epsilon_{C,x}>\frac dm$.  By  Corollary \ref{cor: dC} we have  $dC \in |kL|$ and $dC$ is submaximal as well.  Then by Bauer's bound, Theorem \ref{thm: Bauerbound}, we have
 \[k = \frac 1d k L^2 = L.C<\frac{d}{\sqrt d - \epsilon_{C,x}} <  m .\]
This inequality implies by the projection formula
\[H^0(X, kL) \isom \bigoplus_{i=0}^{d-1}H^0(\IP^2, (k-im)H) = H^0(\IP^2, kH), \]
that is, every curve of degree at most $m$ is pullback of a plane curve. 
By  Lemma \ref{lem: no pullback} the curve $dC$ is not submaximal --- a contradiction.

Now consider the family $\kx \subset |\ko_{\IP(1,1,1,m)}(md)|$ of smooth hypersurfaces of degree $md$ in $\IP(1,1,1,m)$ not containing the singular point and let
\[
\begin{tikzcd}\ku \rar[hookrightarrow]\dar & \IP(1,1,1,m) \times \kx \\  \kx
 \end{tikzcd}
\]
be the universal family. 
Let $\tilde\ku\to \ku\times_\kx \ku$ be the blow-up of the diagonal and $E$ the exceptional divisor of the blow up. 
Then 
\[ f \colon \tilde \ku \to \ku\]
is a smooth family of surfaces. If we consider a point in $\ku$ as a pair $(X,x)$ of a hypersurface $X$ and a point $x\in X$, then $\tilde X = \inverse f(X,x)$ is exactly the surface $X$ blown up at the point $x$, with exceptional divisor $E_x = E\cap \tilde X$. 
Let $\kl$ be the pullback of $\ko_{\IP(1,1,1,m)}(1)$ to $\ku$; it is a line bundle on every fibre by our assumptions.

Now note that the simple cyclic $d$-uple planes branched over a curve of degree $md$ are contained in this family and in particular, there is a point $(X,x)\in \ku$ such that $\epsilon(\kl|_X;x) \geq \sqrt d -\frac dm$. 
By the characterisation of Seshadri constants via nefness on the blow up in \eqref{eq: Seshadri via nefness} we have that 
\[\left(f^*\kl -\epsilon(\kl|_X;x) E\right)|_{\inverse f(X,x)}\]
is nef. By \cite[Corollary 4]{Moriwaki} there is an open subset $\ku'\subset \ku$ containing $(X,x)$ and such that for every $(Y,y)\in \ku'$ the line bundle $\kl|_Y-\epsilon(\kl|_X;x)E_y$ is nef. 

Using \eqref{eq: Seshadri via nefness} again, we see that $\epsilon(\kl|_X;x)\leq \epsilon(\kl|_Y;y)$. Therfore the estimate \eqref{eq: estimate} holds for all pairs in the open subset $\ku'$.
This concludes the proof of Theorem~A.  \qed

\end{document}